\documentclass[12pt,oneside,reqno]{amsart}
\usepackage[all]{xy}
\usepackage{amsfonts,amsmath,oldgerm,amssymb,amscd}
\UseComputerModernTips
\numberwithin{equation}{section}
\usepackage[breaklinks]{hyperref}
\allowdisplaybreaks

\usepackage{caption} 
\newtheorem{theorem}{Theorem}
\newtheorem{proposition}[theorem]{Proposition}

\newtheorem{lemma}[subsection]{{\bf Lemma}}

\newtheorem{remark}[subsection]{Remark}

\newcommand{\al}{\alpha}
\newcommand{\be}{\beta}
\newcommand{\ga}{\gamma}

\newcommand{\Q}{\mbox{$\mathbb Q$}}

\begin{document}
\title[Prime powers in sums of terms of binary recurrence sequences]{Prime powers in sums of terms of binary recurrence sequences} 

\author[Mazumdar]{Eshita Mazumdar}
\address{Eshita Mazumdar,  Departament of Mathematics, Indian Institute of Technology Bombay\\
Powai, Mumbai-400 076, India}
\email{tuli.mazumdar@gmail.com}

\author[Rout]{S. S. Rout}
\address{Sudhansu Sekhar Rout, Departament of Mathematics\\ Harish-Chandra Research Institute\\ Chhatnag Road, Jhunsi\\ India, 211019}
\email{lbs.sudhansu@gmail.com}

\maketitle
\pagenumbering{arabic}
\pagestyle{headings}

\begin{abstract}
Let $\{u_{n}\}_{n \geq 0}$ be a non-degenerate binary recurrence sequence with positive discriminant and $p$ be a fixed prime number. In this paper, we have shown a finiteness result for the solutions of the Diophantine equation $u_{n_{1}} + u_{n_{2}} + \cdots + u_{n_{t}} = p^{z}$ with $n_1 \geq n_2 \geq \cdots \geq n_t\geq 0$. Moreover, we explicitly find all the powers of three which are sums of three balancing numbers using the lower bounds for linear forms in logarithms. Further, we use a variant of Baker-Davenport reduction method in Diophantine approximation due to Dujella and Peth\H{o}.
 
\end{abstract}

\section{Introduction}
There are many interesting Diophantine equations arises when one study the intersection of two sequences of positive integers. More precisely, one can ask when the terms of a fixed binary recurrence sequence be perfect powers, factorials or a combinatorial numbers etc. For example, one can consider the solvability of the Diophantine equation 
\begin{equation}\label{eq1}
u_{n}  = x^{z}
\end{equation}
in integers $n,x,z$ with $z\geq 2$, where $\{u_{n}\}_{n\geq 0}$ is a linear recurrence sequence.  Peth\H{o} \cite{Petho1982} and Shorey-Stewart \cite{Shorey1983} independently proved under certain natural assumptions that \eqref{eq1} contain only finitely many perfect powers. The problem of finding all perfect powers in Fibonacci sequence  has a very rich history \cite{Cohn1964,Petho1983} and this problem was quite open for a long time. In 2006, Bugeaud, Mignotte, and Siksek \cite{Bugeaud2006} proved that $(n,x,z) \in \{(0, 0, z), (1, 1, z), (2, 1, z), (6, 2, 3), (12, 12, 2)\}$ are the only solutions of \eqref{eq1} when $u_{n}$ is the Fibonacci sequence, using both classical and the modular approach. In the same paper, they also studied the equation \eqref{eq1} for the Lucas number sequence. In 2008, A. Peth\H{o} \cite{Petho1991} (see also \cite{Cohn1996}) solved the equation \eqref{eq1} when $u_{n}$ is the Pell sequence and proved that $0,1$ and $169$ are the only perfect powers.

\smallskip
Now one can extend \eqref{eq1} by taking two terms of a binary recurrence sequence and asks the same question. Recently, several authors studied the problem to find $(n, m, z)$ such that
\begin{equation}\label{eq2}
u_{n} + u_{m} = 2^{z},
\end{equation}
where $\{u_{n}\}_{n\geq 0}$ is a fixed recurrence sequence. In particular, Bravo and Luca considered the case when $u_{n}$ is the Fibonacci sequence \cite{BL2015} and the Lucas number sequence \cite{Bravo2014} respectively. Variants of the equation \eqref{eq2},  was studied independently by Bravo et al. \cite{BravoLuca2016} and Marques \cite{Marques2014}, when $u_{n}$ is replaced by the generalized Fibonacci sequence. In \cite{Bravo2015}, Bravo et al., investigated when power of two can be expressed as sum of three Fibonacci numbers. Recently, Pink and Ziegler \cite{Pink2016} generalize the results due to Bravo and Luca \cite{Bravo2014,BL2015} and consider the more general Diophantine equation 
\begin{equation}\label{eq2a}
u_{n} + u_{m} = wp_{1}^{z_{1}} \cdots p_{s}^{z_{s}}
\end{equation}
in non-negative integer unknowns $n, m, z_{1}, \ldots, z_{s}$, where $\{u_{n}\}_{n\geq 0}$ is a binary non-degenerate recurrence sequence, $p_{1}, \ldots, p_{s}$ are distinct primes and $w$ is a non-zero integer with $p_{i} \nmid w$ for $1\leq i \leq s$. They proved, under certain assumptions, \eqref{eq2a} has finitely many solutions using lower bounds for linear forms of $p$-adic logarithms. Recently in \cite{Bertok2017}, a similar type of equation has been studied when the right hand side of \eqref{eq2a} is replaced by linear combinations of prime powers.

\smallskip
The purpose of this paper is twofold. On one hand,  we give a general finiteness result for the solutions of the equation 
\begin{equation}\label{eq3}
u_{n_{1}} + u_{n_{2}} + \cdots + u_{n_{t}} = p^{z}
\end{equation}
in non-negative integer unknowns $n_{1},\ldots, n_{t}, z$, where  $\{u_{n}\}_{n\geq 0}$ is a binary non-degenerate recurrence sequence with $n_{1} \geq n_{2} \geq \cdots \geq n_{t} \geq 0$ and $p$ is a given prime. On the
other hand, we completely solve the equation \eqref{eq3} when $u_{n}$ is a balancing number sequence and $(t, p) = (3, 3)$. To prove our main theorems, we use lower bounds for linear forms in logarithms of algebraic numbers and a version of the Baker-Davenport reduction method.

\section{Notations and Main Results} 
The sequence $\{u_{n}\}_{n \geq 0} = \{u_{n}(P, Q, u_{0}, u_{1})\}$ is called a binary linear recurrence sequence if the relation
\begin{equation}\label{eq4}
u_{n} = Pu_{n-1} + Qu_{n-2} \;\;(n\geq 2)
\end{equation}
holds, where $PQ\neq 0, u_{0}, u_{1}$  are fixed rational integers and $|u_{0}| + |u_{1}| > 0$. Then for $n\geq 0$
\begin{equation}\label{eq5}
u_{n} = \frac{a\alpha^{n}-b\beta^{n}}{\alpha-\beta}\quad (\alpha \neq \beta), 
\end{equation}
where $\alpha$ and $\beta$ are the roots of the polynomial $x^2-Px-Q$ and $a = u_{1} - u_{0}\beta,\, b = u_{1} - u_{0}\alpha$. The
sequence $\{u_{n}\}$ is called non-degenerate, if $ab\al\be \neq 0$ and $\al/\be$ is not a root of unity.  

Throughout the paper, we assume that $u_{n}$ is non-degenerate, $\sqrt{\Delta} = (\alpha-\beta) >0$. The latter assumption implies that the sequence $\{u_{n}\}_{n \geq 0}$ have a dominant root and hence we can assume that $|\al| > |\be|$. With these notations, we have the following theorem.
\begin{theorem}\label{th1}
Let $\{u_{n}\}_{n \geq 0}$ be a non-degenerate binary recurrence sequence with $\Delta > 0$. Then there exists an effectively computable constant $C$ depending on $\{u_{n}\}_{n \geq 0}, p, t$ such that all solutions $(n_{1},\ldots, n_{t},z)$ to equation \eqref{eq3} satisfy 
\[\max\{n_1,\ldots,n_t,z\} < C.\]
\end{theorem}

\smallskip 
{\it Balancing numbers} $n$ are solutions of the Diophantine equation 
\begin{equation*}
\sum_{i=1}^{n-1}i=\sum_{j=n+1}^{m}j,
\end{equation*} 
for some natural number $m$ \cite{Behera1999}. Let us denote the $n$-th balancing number by $B_{n}$. Also, the balancing numbers satisfy the recurrence relation $B_{n+1}=6B_{n}-B_{n-1}$ with initial conditions $B_0 = 0, B_1 = 1$ for $n\geq 1$. Therefore, 
the sequence \eqref{eq4} for $(P, Q, u_{0}, u_{1}) = (6, -1, 0, 1)$
is a balancing sequence. For more details regarding balancing numbers one can refer \cite{Behera1999, Rout2015}. We prove the following theorem as an example for the explicit computations of constants.
\begin{theorem}\label{th2}
The only solutions of the Diophantine equation 
\begin{equation}\label{eq7}
B_{n_{1}} + B_{n_{2}} + B_{n_{3}} = 3^{z}
\end{equation}
in integers $n_{1}, n_{2}, n_{3}, z$ with $n_{1}\geq n_{2} \geq n_{3} \geq 0$ are   $(n_{1}, n_{2}, n_{3}, z) \in \{(1, 1, 1, 1),(1,0,0,0)\}.$
\end{theorem}

\section{Auxiliary results}
In the following lemma we find an upper bound for $|u_{n}|$ which will be useful in the proof of Theorem \ref{th1} and also we find a relation between $n_1$ and $z$.
\begin{lemma}\label{lem8a}
There exists constants $d_0$ and $d_1$ such that following holds.
\begin{enumerate}
\item $|u_{n}|\leq d_{0}|\al|^{n}$.
\item If \eqref{eq3} holds, then $z\leq d_1 n_1$. 
\end{enumerate}
\end{lemma}
\begin{proof}
From \eqref{eq5}, we have 
\begin{align*}
|u_{n}| & \leq \frac{|\al|^{n}}{\sqrt{\Delta}}\left(|a| + |b| |\frac{\be^{n}}{\al^{n}}|\right).
\end{align*}
Since $|\al| > |\be|$, the above inequality becomes 
\begin{equation}\label{eq8}
|u_{n}|\leq |\al|^{n} \frac{|a| + |b|}{\sqrt{\Delta}} = d_{0}|\al|^{n},
\end{equation}
where $d_{0} := (|a| + |b|)/\sqrt{\Delta}$. This proves (1). 

\smallskip

From \eqref{eq3} and \eqref{eq8}
\[p^z < d_0(|\alpha|^{n_1} + \cdots + |\alpha|^{n_t}).\]
As $p$ is given, we always find a constant $d_1$ such that $d_{0}|\al|^{n} < p^{nd_1}$. Thus, the above inequality becomes 
\begin{align*}
p^z &< \left(p^{n_1d_1} + \cdots + p^{n_{t}d_{1}}\right)\\
& <p^{n_1d_1}\left(\frac{p}{p-1}\right) \leq p^{d_{1}n_{1}+1}. 
\end{align*}
Therefore, we get $z \leq d_{1}n_{1}$.
\end{proof}

The following lemma is due to Peth\H{o} and de Weger \cite{Petho1986}. 
\begin{lemma}[\cite{Petho1986}]\label{lem9}
Let $u, v \geq 0, h\geq 1$ and $x\in \mathbb{R}$ be the largest solution of $x = u + v (\log x)^{h}$. Then
\[x < \max\{2^{h}(u^{1/h} + v^{1/h} \log(h^{h}v))^{h}, 2^{h}(u^{1/h} + 2e^{2})^{h}\}.\]
\end{lemma}
Let $\eta$ be an algebraic number of degree $d$ with minimal polynomial 
\[a_{0}x^d + a_1x^d + \cdots + a_d = a_0 \prod_{i=1}^{d}\left(X - \eta^{(i)}\right)\]
where the $a_i$'s are relatively prime integers with $a_0 > 0$ and the
$\eta^{(i)}$'s are conjugates of $\eta$. We have defined the absolute {\it logarithmic height} of an algebraic number $\eta$ by
\[
h(\eta) = \frac{1}{d} \left( \log |a_0| + \sum_{i=1}^d \log \max ( 1, |\eta^{(i)}| ) \right). 
\]
In particular, if $\eta = p/q$ is a rational number with $\gcd(p, q) = 1$ and $q >0$, then $h(\eta) = \log \max\{|p|, |q|\}$. Here are some important properties of logarithmic height which we use in our further investigation. 
\begin{enumerate}
\item $h(\eta \pm \gamma) \leq h(\eta) + h(\gamma) + \log 2$,
\item $h(\eta\gamma^{\pm 1}) \leq h(\eta) + h(\gamma)$,
\item $h(\eta^{t}) \leq |t|h(\eta),$\; for $t \in \mathbb{Z}$.
\end{enumerate}

To prove our theorem, we use lower bounds for linear forms in logarithms to bound the index $n_1$ appearing in \eqref{eq3}. 
We need the following general lower bound for linear forms in logarithms due to Matveev \cite{Matveev2000} (see also \cite[Theorem 9.4]{Bugeaud2006}).
\begin{lemma}[Matveev \cite{Matveev2000}]\label{lem12}
Let $\ga_1,\ldots,\ga_t$ be real algebraic numbers and let $b_{1},\ldots, b_{t}$ be non-zero rational integers. Let $D$ be the degree of the number field $\mathbb{Q}(\ga_1,\ldots,\ga_t)$ over $\mathbb{Q}$ and let $A_{j}$ be real number satisfying 
\begin{equation}\label{eq8a}
A_j \geq \max \left\{ Dh(\ga_j) , |\log \ga_j|, 0.16  \right\}, \quad j= 1, \ldots,t.
\end{equation}
 
Assume that $B\geq \max\{|b_1|, \ldots, |b_{t}|\}$ and $\Lambda:=\ga_{1}^{b_1}\cdots\ga_{t}^{b_t} - 1$. If $\Lambda \neq 0$, then
\[|\Lambda| \geq \exp \left( -1.4\times 30^{t+3}\times t^{4.5}\times D^{2}(1 + \log D)(1 + \log B)A_{1}\cdots A_{t}\right).\]
\end{lemma}
To reduce the upper bound which is generally too large, we need a variant of the Baker-Davenport Lemma, which is due to Dujella and Peth\H{o} \cite{Dujella1998}. Here, for a real number $x$, let $||x|| := \min \{|x - n| : n \in \mathbb{Z}\}$ denote the distance from $x$ to the nearest integer.
\begin{lemma}[\cite{Dujella1998}]\label{lem13}
Suppose that $M$ is a positive integer, and $A, B$ are positive reals with $B > 1$. Let $p/q$ be the convergent of the continued fraction expansion of the irrational number $\ga$ such that $q > 6M$, and let  $\epsilon := ||\mu q|| - M||\ga q||$, where $\mu$ is a real number. If  $\epsilon > 0$, then there is no solution of the inequality
\[0 < u \ga - n + \mu < AB^{-m}\]
in positive integers $u, m$ and $n$ with
\[u \leq M\quad \mbox{and}\quad m\geq \frac{\log (Aq/\epsilon)}{\log B}.\]
\end{lemma}

To proof Theorem \ref{th1}, we apply linear forms in logarithms $t$ times. Every time we find an upper bound for $(n_1 - n_i)$ in terms of $n_1$ for all $2\leq i \leq t$. In order to apply the linear forms in logarithms to bound  $(n_1 - n_i)$ for a fixed $i$ we require upper bounds of $(n_1 - n_j)$ for all $1\leq j < i$. Finally, using these upper bounds for $(n_1 - n_i),\;2\leq i \leq t$, we get an upper bound for $n_1$. In order to apply Matveev's Theorem we must ensure that $\Lambda$ does not vanish. In this regard, we have the following lemma.
\begin{lemma}\label{lem13a}
Suppose $\Lambda_i := p^{z}\alpha^{-n_1} a^{-1}\sqrt{\Delta}(1 + \al^{n_2 - n_1} + \cdots + \al^{n_i - n_1})^{-1} - 1$ for all $1\leq i \leq t$. If  $\Lambda_i = 0$, then $n_1 \leq \frac{\log ((|a|(t - i)+(t - 1)|b|)/|b|)}{
\log (|\beta|/|\alpha|)}$.
\end{lemma}
\begin{proof}
 $\Lambda_i = 0$  imply 
\begin{equation}\label{eq17}
p^z \sqrt{\Delta} =a \al^{n_1}(1 + \al^{n_2 - n_1} + \cdots + \al^{n_i - n_1}). 
\end{equation}
From \eqref{eq3} and \eqref{eq17}, we have 
\[a(\alpha^{n_1}+\cdots+\alpha^{n_t}) - 
b(\beta^{n_1}+\cdots+\beta^{n_t})= a\alpha^{n_1}+ \cdots+
a\alpha^{n_i}\]
and this implies that $a(\alpha^{n_{i+1}}+\cdots+\alpha^{n_t})
= b(\beta^{n_1}+\cdots+\beta^{n_t})$. But $|a\alpha^{n_1}(t-i)| \geq |a(\alpha^{n_{i+1}}+\cdots+\alpha^{n_t})|$, thus
\[|a\alpha^{n_1}(t-i)| \geq |b(\beta^{n_1}+\cdots+\beta^{n_t})| \geq |b\beta^{n_1}|
- |b(\beta^{n_2}+\cdots+\beta^{n_t})|.\]
As $|\alpha| > |\beta|$, we deduce that 
\[|b\beta^{n_1}| \leq |a\alpha^{n_1}(t-i)| + |b \alpha^{n_1}
|(t-1) \leq |\alpha^{n_1}|(|a|(t-i)+|b|(t-1)).\]
Hence 
\[n_1 \leq \frac{\log ((|a|(t - i)+(t - 1)|b|)/|b|)}{
\log (|\beta|/|\alpha|)}.\]
\end{proof}
The following lemma gives a relation between height of an algebraic number and its logarithm.
 
\begin{lemma}\label{lem13b}
 Let $\gamma_3(i) := a^{-1}\sqrt{\Delta}(1 + \al^{n_2 - n_1} + \cdots + \al^{n_i - n_1})^{-1}$ are the algebraic numbers in the number field $\Q(\sqrt{\Delta})$ for $1 \leq i \leq t$. Then there exists a constant $d_2$ depending on $\al$ and $\beta$ such that
\begin{align*}\label{eq19}
A_3(i): = &2\left(\log|a| + d_2 \log  \max \{\sqrt{\Delta}, 1/\sqrt{\Delta}\}
+ |\log \sqrt{\Delta}| + \left(|n_2 - n_1| + \cdots + |n_i - n_1| \right) h(\al)\right) \nonumber \\
& + (i + 1)\log 4 \geq \max \left\{ Dh(\ga_3(i) , |\log \ga_3(i)|, 0.16  \right\}.
\end{align*}

\end{lemma}
\begin{proof}
Since $\gamma_3(i) = a^{-1}\sqrt{\Delta}(1 + \al^{n_2 - n_1} + \cdots + \al^{n_i - n_1})^{-1}$ and $|\al| > 1$, we have
\[\gamma_3(i) = \frac{\sqrt{\Delta}}{a(1 + \al^{n_2 - n_1} + \cdots + \al^{n_i - n_1})} \leq |a|\sqrt{\Delta}|(1 + \al^{n_2 - n_1} + \cdots + \al^{n_i - n_1})|\leq i|a|\sqrt{\Delta}\]
and 
\[\gamma_3(i)^{-1} = \frac{a(1 + \al^{n_2 - n_1} + \cdots + \al^{n_i - n_1})}{\sqrt{\Delta}} \leq i|a|\sqrt{\Delta}.\]
As $\log i < i \log 2$ for $i \geq 1$, we have 
\begin{equation}\label{eq20}
|\log \gamma_3(i)| \leq \log |a| + \log \sqrt{\Delta} + i\log 2.
\end{equation} 
Now we estimate height of $\ga_3(i)$. We have
\begin{align}\label{eq20a}
\begin{split}
h(\ga_3(i)) &= h\left(a^{-1}\sqrt{\Delta}(1 + \al^{n_2 - n_1} + \cdots + \al^{n_i - n_1})^{-1}\right)\\
& \leq \log|a| + h(\sqrt{\Delta}) + \left(|n_2 - n_1| + \cdots + |n_i - n_1| \right) h(\al) + i\log 2
\end{split}
\end{align}
Now suppose that $\Delta$ is squarefree. Then
\begin{equation}\label{eq20b}
h(\sqrt{\Delta}) = \log \max \{1, \sqrt{\Delta}\} \leq |\log \sqrt{\Delta}|. 
\end{equation} 
Now assume that $\Delta$ is perfect square and let $\Delta = s^2$, where $s= \frac{p}{q}$ with $\gcd(p,q)=1$ and
$p,q>0$. If $p = q$, then $h(\sqrt{\Delta}) = 0$. Now suppose that $p\neq q$. Thus,
\begin{equation}\label{eq20c}
h(\sqrt{\Delta}) = h(p/q)= \log \max\{p,q\}= \log \max\{\sqrt{
\Delta}q, q\}= \log q + \log \max \{1, \sqrt{\Delta}\}.
\end{equation}
 On the other hand, $\max \{\sqrt{\Delta}, \frac{1}{\sqrt{\Delta}}\} > 1$.
Therefore, there exists a constant $d_2$ such that $q < \max \{\sqrt{\Delta}, \frac{1}{\sqrt{\Delta}}\}^{d_2}$ which implies that $\log q < d_2 \log  \max \{\sqrt{\Delta}, \frac{1}{\sqrt{\Delta}}\}$. Thus from \eqref{eq20c}, we obtain
\begin{equation}\label{eq20d}
h(\sqrt{\Delta})< d_2 \log  \max \{\sqrt{\Delta}, \frac{1}{\sqrt{\Delta}}\}
+|\log \sqrt{\Delta}|.
\end{equation}
Hence the lemma follows by comparing \eqref{eq20}, \eqref{eq20a},  \eqref{eq20b} and \eqref{eq20d}.
\end{proof}
Suppose $\ell_i:= \frac{\log ((|a|(t - i)+(t - 1)|b|)/|b|)}{
\log (|\beta|/|\alpha|)}$ (see Lemma \ref{lem13a}). We denote $\ell:= \max\{\ell_1, \cdots, \ell_{t}\}$.
If $n_1 \leq \ell$, then the conclusion of Theorem \ref{th1} follows trivially. Here on out, we assume $n_1 > \ell$.

Now we are ready to prove Theorem \ref{th1}. The proof is  motivated by Bravo and Luca \cite{BL2015}.

\section{Proof of Theorem \ref{th1}}
Firstly, observe that if $n_t = 0$ and $(t-1) \geq 2$, then it is equivalent to consider \eqref{eq3} again. Now if $n_t = 0$ and $t = 2$ then \eqref{eq3} reduces to $u_{n_{1}} = p^{z}$ and this equation has been considered independently by Peth\H{o} and Shorey-Stewart (see \cite{Petho1982, Shorey1983}). Suppose $n_1 = \cdots = n_t$, then \eqref{eq3} becomes $t u_{n_1} = p^z$, which is not true. From now on, we assume $n_1 > n_2 > \cdots > n_{t}$ and all the remaining cases can be handled from these cases. Indeed, if some of the $n_i$'s are equal, we can group them together obtaining a representation of the from
\[a_1 u_{n_1} + a_2 u_{a_{1} + 1} + \cdots + a_r u_{a_{r - 1} + 1 } = p^{z}\]
 with $a_1 + \cdots + a_r = t$ and this equation can be handled similar to that of \eqref{eq3} with some changes in constants.

\subsection{Bounding $(n_1 - n_2)$ in terms of $n_1$}\label{sec4.1}
 We can rewrite \eqref{eq3} as
\begin{align*}
\left|\frac{a\al^{n_1}}{\sqrt{\Delta}} - p^{z}\right| & = \left|\frac{b\be^{n_1}}{\sqrt{\Delta}} - (u_{n_2} + \cdots + u_{n_t})\right|\leq \left|\frac{b\be^{n_1}}{\sqrt{\Delta}}\right| + |u_{n_2} + \cdots + u_{n_t}|.
\end{align*}
From Lemma \ref{lem8a}, we have $|u_{n}|\leq d_{0}|\al|^{n}$. Also we have assumed that $n_1 > n_2 > \cdots > n_{t}$. Thus,
\begin{equation}\label{al14}
\left|\frac{a\al^{n_1}}{\sqrt{\Delta}} - p^{z}\right| \leq \left|\frac{b\be^{n_1}}{\sqrt{\Delta}}\right| + (t-1)d_{0}|\al|^{n_{2}}.
\end{equation}

If $|\beta| < 1$, then the above inequality becomes
\begin{equation}\label{eq141}
\left|\frac{a\al^{n_1}}{\sqrt{\Delta}} - p^{z}\right| < c_{1}|\al|^{n_{2}}.
\end{equation}
Now dividing both sides of the inequality \eqref{eq141} by $a\al^{n_1}/\sqrt{\Delta}$,
\begin{align}\label{eq14}
\begin{split}
\left|1 - p^{z}\alpha^{-n_1} a^{-1}\sqrt{\Delta} \right| & \leq  \frac{c_{1}\sqrt{\Delta}}{a|\alpha|^{n_1-n_2}} < \frac{c_2}{|\alpha|^{n_1-n_2}}.
\end{split}
\end{align}
Now suppose $|\beta| > 1$. Again dividing both sides of the inequality  \eqref{al14} by $a\al^{n_1}/\sqrt{\Delta}$,
\begin{align*}
\left|1 - p^{z}\alpha^{-n_1} a^{-1}\sqrt{\Delta} \right| & \leq \left|\frac{b\be^{n_1}}{a \al^{n_1}}\right| + \frac{(t-1)d_{0}\sqrt{\Delta}}{|a|}|\al|^{n_{2} -n_1}\\
& < \frac{|b|}{|a|}\left(\frac{|\al|}{|\be|}\right)^{n_2 - n_1} + \frac{(t-1)d_{0}\sqrt{\Delta}}{|a|}|\al|^{n_{2} -n_1}\\
& < \frac{|b|}{|a|}\left(\frac{|\be|}{|\al|}\right)^{n_1 - n_2} + \frac{(t-1)d_{0}\sqrt{\Delta}}{|a|}\left(\frac{|\be|}{|\al|}\right)^{n_{1} -n_2}\\
& < \left(\frac{|b|}{|a|} +  \frac{(t-1)d_{0}\sqrt{\Delta}}{|a|}\right)\left(\frac{|\be|}{|\al|}\right)^{n_{1} -n_2}.
\end{align*}
Thus, for any $\beta$ it follows that
\begin{equation}\label{eq14a}
\left|1 - p^{z}\alpha^{-n_1} a^{-1}\sqrt{\Delta} \right| < \frac{c_3}{\min \left(\frac{|\al|}{|\be|}, |\al|\right)^{n_{1} -n_2}} 
\end{equation}
where $c_3 := \max\left\{c_2, \left(\frac{|b|}{|a|} +  \frac{(t-1)d_{0}\sqrt{\Delta}}{|a|}\right)\right\}$.  
In order to apply Lemma \ref{lem12}, we take 
\[\gamma_1 := p,\; \gamma_2 := \alpha,\; \gamma_3 := a^{-1}\sqrt{\Delta},\; b_1 := z,\; b_2 := -n_1,\; b_3= 1.\] 
Thus our first linear form is $\Lambda_1 := \ga_{1}^{b_{1}}\ga_{2}^{b_{2}}\ga_{3}^{b_{3}} - 1$ and  $\Lambda_1 \neq 0$ from Lemma \ref{lem13a}.  Here we are taking the field $\Q(\sqrt{\Delta})$ over $\Q$ and $t=3$. Finally, we recall that $z\leq d_1n_1$ and deduce that 
\[\max\{|b_1|, |b_2|, |b_3|\} = \max\{z, n_1, 1\} \leq d_1n_1.\]
Hence we can take $B := d_1n_1$. Also $D\leq 2, h(\gamma_1)= \log p, h(\gamma_2)\leq \log \alpha$. Thus, we can take \[A_1 := 2\log p, A_2 = 2\log |\alpha|, A_3 = 2\left(\log |a| + d_2 \log  \max \{\sqrt{\Delta}, 1/\sqrt{\Delta}\}
+ |\log \sqrt{\Delta}| + \log 4 \right).\]  From Lemma \ref{lem12}, we have 
\begin{align*}
|\Lambda_1| & > \exp \left(-C_0(1+ \log d_1n_1)\right),
\end{align*}
where $C_0:= 1.4\times 30^{5}\times 2^{4.5}\times 4\times (1+\log D)(2\log p) (2\log |\alpha|) A_3$. So the above inequality can be rewritten as, 
\begin{equation}\label{eq15}
\log |\Lambda_1| > - C_{1} \log n_1. 
\end{equation} 
Taking logarithms in inequality \eqref{eq14a} and comparing the resulting inequality with \eqref{eq15}, we get that
\begin{equation}\label{eq15a}
(n_1-n_2)\log \max \left(\frac{|\al|}{|\be|}, |\al|\right) < C_2 \log n_1.
\end{equation}

\subsection{Bounding $(n_1 - n_3)$ in terms of $n_1$} \label{sec4.2}
To formulate the second linear form we rewrite \eqref{eq3} as follows
\begin{equation*}
\frac{a\al^{n_1}}{\sqrt{\Delta}} + \frac{a\al^{n_2}}{\sqrt{\Delta}} - p^{z}  = \frac{b\be^{n_1}}{\sqrt{\Delta}} + \frac{b\be^{n_2}}{\sqrt{\Delta}} - (u_{n_3} + \cdots + u_{n_t}).
\end{equation*}
Taking absolute value with $|\be|\leq 1$ and using \eqref{eq8}, we have
\begin{align*}
\left|\frac{a\al^{n_1}}{\sqrt{\Delta}}(1 + \al^{n_2 - n_1}) - p^{z}\right|& \leq \left|\frac{b(\be^{n_1} + \be^{n_2})}{\sqrt{\Delta}}\right| +  (t-2)d_{0}|\al|^{n_{3}}\\
& < c_4|\al|^{n_{3}}.
\end{align*}
Dividing both sides of the inequality by $a(\al^{n_1} + \al^{n_2})/\sqrt{\Delta}$,
\begin{equation}\label{eq16}
\left|1 - p^{z}\alpha^{-n_1} a^{-1}\sqrt{\Delta}(1 + \al^{n_2 - n_1})^{-1} \right|  < \frac{c_{5}}{|\alpha|^{n_1-n_3}}.
\end{equation}
Similarly for $|\be| > 1$, we will proceed like Section \ref{sec4.1}. Thus, for any $\beta$ it follows that
\begin{equation}\label{eq16a}
\left|1 - p^{z}\alpha^{-n_1} a^{-1}\sqrt{\Delta}(1 + \al^{n_2 - n_1})^{-1} \right| < \frac{c_6}{\min \left(\frac{|\al|}{|\be|}, |\al|\right)^{n_{1} -n_3}} 
\end{equation}
where $c_6 := \max\left\{c_5, \left(\frac{|b|}{|a|} +  \frac{(t-2)d_{0}\sqrt{\Delta}}{|a|}\right)\right\}$. 
To apply Lemma \ref{lem12} we take,
 \[\gamma_1 := p,\; \gamma_2 := \alpha,\; \gamma_3 := a^{-1}\sqrt{\Delta}(1 + \al^{n_2 - n_1})^{-1},\; b_1 := z, b_2 := -n_1, b_3= 1.\] 
Therefore,  $\Lambda_2 := \ga_{1}^{b_{1}}\ga_{2}^{b_{2}}\ga_{3}^{b_{3}} - 1$. 
Also, $\Lambda_2 \neq 0$ from Lemma \ref{lem13a}. As $A_1, A_2$ are already estimated in previous case, here we only estimate $A_{3}$. Using Lemmas \ref{lem13b} we can take $A_3 = A_3(2)= c_{7} + 2(n_1 - n_2)\log|\al|$, with \[c_7 =  2\left(\log|a| + d_2 \log  \max \{\sqrt{\Delta}, 1/\sqrt{\Delta}\}
+ |\log \sqrt{\Delta}|\right) + 3\log 4.\]  Again, from Lemma \ref{lem12} and equation \eqref{eq16}, we have
\begin{equation}\label{eq22}
\exp \left(-c_{8}\log n_{1}(c_{7} + 2(n_1 - n_2)\log|\al| \right) < \frac{c_{6}}{\min \left(\frac{|\al|}{|\be|}, |\al|\right)^{n_{1}-n_{3}}}.  
\end{equation}
Now using \eqref{eq15a} in \eqref{eq22}, we get 
\begin{equation}\label{eq23}
(n_1-n_3)\log \max \left(\frac{|\al|}{|\be|}, |\al|\right) < C_3 (\log n_1)^2.
\end{equation}

\begin{remark}
For getting an upper bound of $(n_1-n_3)$ in \eqref{eq23} we use the upper bound of $(n_1-n_2)$.  That's why we have to apply linear forms in logarithms $t$ times. Hence, for $3 \leq i \leq t$ we rewrite \eqref{eq3} in the following form 
\begin{equation*}
\left|\frac{a\al^{n_1}}{\sqrt{\Delta}} + \cdots + \frac{a\al^{n_i}}{\sqrt{\Delta}} - p^{z}\right|  = \left|\frac{b\be^{n_1}}{\sqrt{\Delta}} + \cdots + \frac{b\be^{n_i}}{\sqrt{\Delta}} - (u_{n_{i+1}} + \cdots + u_{n_{t}})\right|
\end{equation*}
 and proceed as above, there exist constants $C_{i}$ such that 
\begin{equation}\label{eq24}
(n_1-n_i)\log \max \left(\frac{|\al|}{|\be|}, |\al|\right) < C_i (\log n_1)^{i-1}.
\end{equation}
\end{remark}
\subsection{Bounding  $n_1$} \label{sec4.3}
To bound $n_1$, we write the equation \eqref{eq3} as  
\begin{equation*}
\frac{a\al^{n_1}}{\sqrt{\Delta}} + \cdots + \frac{a\al^{n_t}}{\sqrt{\Delta}} - p^{z}  = \frac{b\be^{n_1}}{\sqrt{\Delta}} + \cdots + \frac{b\be^{n_t}}{\sqrt{\Delta}}, 
\end{equation*}
which implies that  
\begin{align*}
\left|\frac{a\al^{n_1}}{\sqrt{\Delta}}(1 + \al^{n_2 - n_1} + \cdots + \al^{n_t - n_1}) - p^{z}\right|& = \left|\frac{b(\be^{n_1} + \cdots +\be^{n_t})}{\sqrt{\Delta}}\right|.
\end{align*}
Now dividing through out by $\frac{a\alpha^{n_1}}{\sqrt{\Delta}}(1+\cdots+ \alpha^{n_t-n_1})$,
\begin{equation}\label{eq25}
\left|1 - p^{z}\alpha^{-n_1} a^{-1}\sqrt{\Delta}(1 + \al^{n_2 - n_1} + \cdots + \al^{n_t - n_1})^{-1} \right|  < \frac{c_9}{\min \left(\frac{|\al|}{|\be|}, |\al|\right)^{n_{1}}}. 
\end{equation}

To apply  Matveev's theorem \ref{lem12}, we take 
\begin{align*}
&\gamma_1 := p, \gamma_2 := \alpha, \gamma_3 := a^{-1}\sqrt{\Delta}(1 + \al^{n_2 - n_1} + \cdots + \al^{n_t - n_1})^{-1}\\
& b_1 := z,\;\; b_2 := -n_1,\;\; b_3= 1.
\end{align*}
Thus the final linear form is $\Lambda_t := \ga_{1}^{b_{1}}\ga_{2}^{b_{2}}\ga_{3}^{b_{3}} - 1$ and is non-zero by Lemma \ref{lem13a}. From the conclusions of Lemma \ref{lem13b}, we can take 
\[A_3 = A_3(t) = c_{14} + 2[(n_1 - n_2) + \cdots + (n_1 - n_t)]\log |\al|,\]
where $c_{14} =  2\left(\log|a| + d_2 \log  \max \{\sqrt{\Delta}, 1/\sqrt{\Delta}\}
+ |\log \sqrt{\Delta}|\right) + (t+1)\log 4$.
Using Lemma \ref{lem12} and \eqref{eq25}, we have
\begin{equation}\label{eq26}
\exp \left( - c_{15}\log n_{1}(c_{14} + 2[(n_1 - n_2) + \cdots + (n_1 - n_t)]\log |\al| \right) < \frac{c_{8}}{\min \left(\frac{|\al|}{|\be|}, |\al|\right)^{n_{1}}}.  
\end{equation}
Hence putting equations \eqref{eq15a},\eqref{eq23} and \eqref{eq24}  in \eqref{eq26}, we obtain 
\begin{equation}\label{eq27}
n_1\log \max \left(\frac{|\al|}{|\be|}, |\al|\right) < C_t (\log n_1)^{t}.
\end{equation}
Theorem \ref{th1} follows by applying Lemma \ref{lem9} and \ref{lem8a}(2) to the inequality \eqref{eq27}.

\section{Proof of Theorem \ref{th2}}
We will give the computational details of the resolution of Diophantine equation \eqref{eq7}. Let $\{B_{n}\}_{n \geq 0}$ be the balancing sequence given by $B_{0} = 0, B_1 = 1$ and $B_{n+1} = 6B_{n} - B_{n-1}$ for all $n\geq 1$. Now one can easily see from \eqref{eq5} that $a =1, b=1, \alpha=3+2\sqrt{2},\; \beta=3-2\sqrt{2}$ and the general terms of balancing numbers are 
\begin{equation}\label{eq6}
B_{n}=\frac{\alpha^{n}-\beta^{n}}{\alpha-\beta} \quad \mbox{for all}\;\; n=0,1,\cdots.
\end{equation}
\subsection{The case $n_3 =0$} If $n_2 = 0$, then \eqref{eq7} becomes $B_{n_1} = 3^{z}$ which is not true (see \cite{Dey}). If $n_2 > 0$, then we have 
\begin{equation}\label{eq28}
B_{n_1} + B_{n_2} = 3^z,
\end{equation}
which is a reduced form of \eqref{eq7}. Hence without loss of generality, from now on we assume $n_3 > 0$.  

\subsection{Bounding $(n_1 - n_2)$ and $(n_1 - n_3)$ in terms of $n_1$} If $n_1 = n_2 = n_3$, the required equation is $3B_{n_1} = 3^z$ and this equation has a trivial solution, i.e., $(n_1,z) = (1,1)$. Therefore, we may consider $n_1 > n_2$ or $n_1 > n_3$.  If $ 1\leq n_1 \leq 100$, then a brute force search with {\it Mathematica} in the range $1\leq n_{3}\leq n_2 < n_1 \leq 100$ gives no solution. Hence from onward, we assume that $n_1 > 100$. From Lemma \ref{lem8a}, for this sequence we can take $d_0= 1, d_1 = 2$ and hence 
\begin{equation}\label{eq30}
z \leq 2n_1.
\end{equation}
Now from \eqref{eq14}, we obtain the first linear form
\begin{equation}\label{eq31}
\left|1 - 3^{z}\alpha^{-n_1} 4\sqrt{2} \right| < \frac{16}{|\alpha|^{n_1-n_2}}.
\end{equation}

In order to apply Lemma \ref{lem12}, we take 
\[t= 3, \;\; \gamma_1 = 3,\;\;\gamma_2= \alpha,\;\;\gamma_3 = 4\sqrt{2}, b_1= z, \; b_2 = -n_1, \; b_3 = 1.\]
Thus, $B = 2n_1, h(\gamma_1)= \log 3= 1.0986,  h(\gamma_2) = (\log \al)/2 = 0.8813, h(\gamma_3) =  \leq 0.8664$. We can choose $A_1 = 2.4, A_2 =  1.9, A_3 =  1.8$. Using these parameters we obtain the lower bound for the linear form $\Lambda_1 := \ga_{1}^{b_{1}}\ga_{2}^{b_{2}}\ga_{3}^{b_{3}} - 1$ is
\[|\Lambda_1| > - 7.9 \times 10^{12}\times (1 + \log 2n_1).\] 
Further, since $(1+\log 2n_1) < 2\log n_1$ for $n_1 > 100$, we get 
\begin{equation}\label{eq32}
(n_1-n_2)\log \al < 15.9\times 10^{12} \log n_1.
\end{equation}
Similarly, proceeding as in Section \ref{sec4.2}, we get the analogues equation of \ref{eq16a} as follows:
\begin{equation}\label{eq33}
| 1- 3^z 4\sqrt{2}\alpha^{-n_1}(1+\alpha^{n_2-n_1})^{-1}|< \frac{7}{\al^{n_1-n_3}}. 
\end{equation}
Here our second linear form is $\Lambda_2 := \ga_{1}^{b_{1}}\ga_{2}^{b_{2}}\ga_{3}^{b_{3}} - 1$ where 
\[t= 3, \;\; \gamma_1 = 3,\;\;\gamma_2 = \alpha,\gamma_3=4\sqrt{2} (1 + \alpha^{n_2- n_1})^{-1},\;\; b_1= z, \; b_2 = -n_1, \; b_3 = 1.\]
In this case $A_1$ and $A_2$ are same as previous case but  $A_3 = \log 8\sqrt{2} + (n_1 - n_2)\log \al $. Thus by applying Lemma \ref{lem12} we get 
\begin{equation}\label{eq34}
|\Lambda_2| > - 4.4 \times 10^{12}\times (1 + \log 2n_1)(\log 8\sqrt{2} + (n_1 - n_2)\log \al).
\end{equation}
From \eqref{eq33} and \eqref{eq34}, 
\begin{equation}\label{eq35}
(n_1-n_3)\log \alpha < 1.4\times 10^{26} (\log n_1)^2.
\end{equation}
\subsection{Bounding  $n_1$} To bound $n_1$, we have considered the analogue equation of \eqref{eq25} as follows
\begin{equation}\label{eq36}
| 1- 3^z 4\sqrt{2}\alpha^{-n_1}(1+\alpha^{n_2-n_1} + \alpha^{n_3-n_1})^{-1}|< \frac{3}{\alpha^{n_1}}. 
\end{equation}
Now our final linear form is $\Lambda_3 := \ga_{1}^{b_{1}}\ga_{2}^{b_{2}}\ga_{3}^{b_{3}} - 1$ where 
\[t= 3, \;\; \gamma_1 = 3,\;\;\gamma_2 = \alpha,\gamma_3=4\sqrt{2} (1+\alpha^{n_2-n_1} + \alpha^{n_3-n_1})^{-1}.\]
Here by taking $A_3 = \log 4\sqrt{2} + (n_1 - n_2)\log \al + (n_1 - n_3)\log \al +  \log 4$, we get 
\begin{equation}\label{eq37}
|\Lambda_3| > - 4.4 \times 10^{12}\times (1 + \log 2n_1)(6 + (n_1 - n_2)\log \al + (n_1 - n_3)\log \al).
\end{equation}
From \eqref{eq36} and \eqref{eq37}, 
\begin{equation}\label{eq38}
n_1 \log \alpha < 22\times 10^{38} (\log n_1)^3.
\end{equation}
Further, from Lemma \ref{lem9}, we have 
\begin{equation}
n_1 < 1.5\times 10^{45}.
\end{equation}
Now we summarize the above discussion in the following proposition. 
\begin{proposition}\label{prop1}
Let us assume that $n_1 \geq n_2 \geq n_3$ and $n_1 > 100$. If
$(n_1, n_2, n_3, z)$ 
is  a positive integral solution of  equation \eqref{eq7},   then
\begin{equation}\label{eq39}
z \leq 2n_1 < 3\times 10^{45}. 
\end{equation}
\end{proposition}
\subsection{Reducing the size of $n_1$}
From Proposition \ref{prop1}, we can see that the bound we have obtained for $n_{1}$ is very large. Now our job is to reduce this upper bound to a certain minimal range. From \eqref{eq31}, put 
\begin{equation}\label{eq40}
\Lambda _1 := z\log3 -n\log\alpha + \log 4\sqrt{2}.
\end{equation}
Then this implies 
\[|1- e^{\Lambda_1}|< \frac{16}{\alpha^{n_1-n_2}}.\] 
Note that $\Lambda_1 > 0$, otherwise $3^z \leq \al^{n_1}/4\sqrt{2}$. But we have always
 $$ \frac{\alpha^{n_1}}{4\sqrt{2}}< B_{n_1}+1\leq B_{n_1}+B_{n_2}+B_{n_3}= 3^a.$$ Hence using the fact that $1+x< e^x $ holds for all positive real numbers $x$, we get  
\[0< \Lambda_1 \leq e ^{\Lambda_1} - 1 < \frac{16}{\al^{n_1-n_2}}.\]
Dividing \eqref{eq40} by $\log \alpha$, we have
\begin{equation}\label{eq41}
0< z\left(\frac{\log3}{\log\alpha}\right)-n +\left(\frac{\log 4\sqrt{2}}{\log \alpha}\right)< \frac{10}{\al^{n_1-n_2}}.
\end{equation}
We are now ready to use Lemma \ref{lem13} with the obvious parameters
\[ \gamma:= \frac{\log3}{\log\alpha},\;\; \mu:= (\frac{\log 4\sqrt{2}}{\log \alpha}),\;\;
 A:= 10,\;\; B:= \alpha.\]
Let $[a_0, a_1, a_2,\ldots] = [0, 1, 1, 1, 1, 1, 8, 4, 17,\ldots]$ be a continued fraction expansion of $\ga$ and let $p_k/q_k$ be its $k$-th convergent. Take $M:= 3\times 10^{45}$, then using {\it Mathematica}, one can see that 
\[ 6M < q_{99}.\]
To apply Lemma \ref{lem13}, consider $\epsilon:= \|\mu q_{99}\| - M\|\ga q_{99}\|$ which is positive. If \eqref{eq7} has a solution $(n_1, n_2, n_3, z)$, then $(n_1 - n_2) \in [0,70]$. Next we look into the equation \eqref{eq33} to estimate the upper bound for $(n_1 - n_3)$. Now put 
\[
\Lambda_2:= z\log 3 - n_1 \log \alpha + \log \phi(n_1-n_2),\] where we take
$\phi (x)= 4\sqrt{2}(1+\alpha^{-x})^{-1}$, which implies
$$| 1- e^{\Lambda_2} | < \frac{7}{\al^{n_1-n_3}}.$$  Using the Binet formula of the
balancing sequence, one can show that $\Lambda_2 >0$ since
$$ \frac{\alpha^{n_1}}{4\sqrt{2}}+\frac{\alpha^{n_2}}{4\sqrt{2}}<
B_{n_1}+B_{n_2}+1\leq B_{n_1}+B_{n_2}+B_{n_3}= 3^a.$$
Altogether we  get 
$$0< \Lambda_2 <\frac{7}{\al^{n_1-n_3}}.$$
Replacing $\Lambda_2$ in the above inequality by its formula and arguing as previously, we get 
\begin{equation}\label{eq42}
0< z\left(\frac{\log3}{\log\alpha}\right)-n_1 +\left(\frac{\log \phi(n_1-n_2)}{\log \alpha}\right) < 
\frac{4}{\al^{n_1-n_3}}.
\end{equation} 
$$  $$
Again we will use here Lemma \ref{lem13} with the following parameters
\begin{equation}\label{eq43}
\gamma:= \frac{\log3}{\log\alpha},\;\; \mu:= \frac{\log \phi(n_1-n_2)}{\log \alpha},\;\; 
 A:= 4,\;\; B:= \alpha.
\end{equation}
Proceeding like before with $M := 3\times 10^{45}$ and applying Lemma \ref{lem13} to the inequality \eqref{eq42} for all possible choices of $n_1-n_2 \in [0,70]$ we find that if \eqref{eq7} has a solution $(n_1, n_2, n_3, z)$, then $(n_1 - n_3) \in [0,72]$. Finally, in order to obtain a better upper bound on $n_1$, we can put 
 $$\Lambda_3:= z\log 3 - n_1 \log \alpha + \log \psi(n_1-n_2, n_1-n_3),$$
 with $\psi(x_1,x_2):= 4\sqrt{2}(1+ \alpha^{-x_1}+\alpha ^{-x_2})^{-1},$ which implies
 $$|1 - e^{\Lambda_3}|< \frac{3}{\al^{n_1}}.$$
 We observed that $\Lambda_3 \neq 0.$ Now we consider  the cases $\Lambda_3 > 0$
 and $\Lambda_3< 0$ separately. If $\Lambda_3 > 0$, then 
 $$ 0< \Lambda_3 < \frac{3}{\al^{n_1}}.$$
 Suppose $\Lambda_3< 0$. Since $\frac{3}{\alpha^{n_1}}< \frac{1}{2}$ for 
 $n_1> 2,$ we get that $|e^{\Lambda_3}-1|< 1/2,$ therefore $e^{|\Lambda_3|}< 2.$ Since $\Lambda_3< 0$, we have that 
 \[ 
 0 < |\Lambda_3| \leq e^{|\Lambda_3|}-1=e^{|\Lambda_3|}|e^{\Lambda_3}-1|<\frac{6}{\al^{n_1}}. \]
Thus for both these cases we have
\begin{equation}\label{eq44}
0 < |\Lambda_3| \leq e^{|\Lambda_3|} - 1< \frac{6}{\al^{n_1}}.
\end{equation}
Putting $\Lambda_3$ in \eqref{eq44} and arguing as previously we obtain
\begin{equation}\label{eq45}
0< \left|z\frac{\log 3}{\log\alpha} - n_1 + \frac{\log \psi(n_1-n_2, n_1-n_3)}{\log\alpha}\right|
< \frac{4}{\al^{n_1}}.
\end{equation}
Now, repeat the same procedure as earlier with $M:= 3\times 10^{45}$ for the inequality \eqref{eq45}. For all possible choices of  $n_1-n_2 \in [0,70]$ and $n_1-n_3 \in [0,72]$, we apply Lemma \ref{lem13} to the inequality \eqref{eq45}. If the equation \eqref{eq7} has a solution $(n_1, n_2, n_3, z)$, then $n_1  \in [0,75]$. This leads to a contradiction to our assumption that $n_1 > 100$, which completes the proof of Theorem \ref{th2}.

\section{Acknowledgement}
The  first author would like to thank Harish-Chandra Research Institute, Allahabad for their warm hospitality during the academic visit. The first  author's research is supported by IIT Bombay Postdoctoral Fellowship and the corresponding author's research is supported by HRI Postdoctoral Fellowship.

\end{document}